\documentclass[a4paper,9pt]{amsart}

\usepackage{amsmath,amssymb,amsfonts,latexsym}
\usepackage[T1]{fontenc} 
\usepackage[latin1]{inputenc} 
\usepackage{enumerate}
\usepackage{hyperref}
\usepackage{enumitem}   
\usepackage{color}
\usepackage[all,cmtip]{xy}
\usepackage[linesnumbered,ruled,vlined]{algorithm2e}
\usepackage{xparse}
\usepackage{stackrel}
\usepackage{parskip}
\usepackage{blkarray}
\usepackage{stmaryrd}
\usepackage{hyperref}

\usepackage{mathrsfs}
\usepackage{url}

\DontPrintSemicolon

\theoremstyle{definition}
\newtheorem{teor}{Theorem}[section]
\newtheorem{prop}[teor]{Proposition}

\newtheorem{lem}[teor]{Lemma}
\newtheorem{defi}[teor]{Definition}

\newtheorem{ej}[teor]{Example}

\newtheorem{rem}[teor]{Remark}
\newtheorem{dyn}[teor]{Definition and Notation}

\newtheorem{Reminder}[teor]{Reminder }

\newcommand\N{\mathbb N}

\newcommand\R{\mathbb R}

\DeclareMathOperator\rk{rank}

\DeclareMathOperator{\spn}{span}

\title[Heuristic Algorithm]{An experimental approach for global polynomial optimization based on moments and semidefinite programming}
\author{Mar\'ia L\'opez Quijorna}

\begin{document}

\begin{abstract}
In this article we provide an experimental algorithm that in many cases gives us an upper bound of the global infimum of a real polynomial on $\R^{n}$. It is very well known that to find the global infimum of a real polynomial on $\R^{n}$, often reduces to solve a 	hierarchy of positive semidefinite programs, called moment relaxations. The algorithm that we present involves to solve a series of positive semidefinite programs whose feasible set is included in the feasible set of a moment relaxation. Our additional constraint try to provoke a flatness condition, like used by Curto and Fialkow, for the computed moments. At the end we present  numerical results of the application of the algorithm to nonnegative polynomials which are not sums of squares. We also provide numerical results for the application of a version of the algorithm based on the method proposed by Nie, Demmel and Sturmfels for the problem of minimizing a polynomial over its gradient variety.
\end{abstract}

\date{  July 6, 2017. }

\subjclass[2010]{Primary:  	90C26,	90C22,  ; Secondary: 44A60, 65F30}

\keywords{Semidefinite programming, Lassere relaxation, polynomial optimization, global optimization.}

\maketitle

\setlength{\parindent}{0pt}

\section{Notation}

Throughout this paper, we suppose $n\in\N=\{1,2,\ldots\}$ and abbreviate $(X_{1},\ldots,X_{n})$ by $\underline{X}$.  We let $\R[\underline{X}]$ denote the ring of real polynomials in n indeterminates. We  denote $\N_{0}:=\N\cup\{0\}$. For $\alpha\in\N^{n}_{0}$, we use the standard notation :
\begin{center}
$|\alpha|:=\alpha_{1}+ \cdots + \alpha_{n} $ and $\underline{X}^{\alpha}:=X_{1}^{\alpha_{1}}\cdots X_{n}^{\alpha_{n}}$
\end{center}
For a polynomial $p\in\R[\underline{X}]$ we denote $p=\sum_{\alpha}p_{\alpha}\underline{X}^{\alpha}$ ($p_{\alpha}\in\R$). For $d\in\N_{0}$, by the notation $\R[\underline{X}]_{d}:=\{\sum_{|\alpha|\leq d}a_{\alpha}\underline{X}^{\alpha}\text{ } |\text{ } a_{\alpha}\in\R \}$  we will refer to the vector space of polynomials with degree less or equal to $d$. Polynomials all of whose monomials have exactly the same degree $d\in\N_{0}$ are called $d$-forms.  They form a finite dimensional vector space that we will denote by:
\begin{equation}
\R[\underline{X}]_{=d}:=\{\sum_{|\alpha|=d}a_{\alpha}\underline{X}^{\alpha}\text{ }|\text{ }a_{\alpha}\in\R \} \nonumber
\end{equation}
so that
\begin{equation}
\R[\underline{x}]_{d}=\R[\underline{X}]_{0}\oplus\cdots\oplus\R[\underline{X}]_{d}.\nonumber
\end{equation}
We will denote by $s_{k}:=\dim \R[\underline{X}]_{k}$ and by $r_{k}:=\dim \R[\underline{X}]_{=k}$. For a matrix $A\in\R^{N\times M}$ we denote by $A_{1},\ldots,A_{M}$ its columns, and we denote $A^{T}$ the transpose matrix. We use the notation $SR^{N\times N}$ to refer us to the vector space of symmetric matrices $N\times N$ with entries in the ring $R$, that is to say:
\begin{equation}
SR^{N\times N}:=\{A\in\R^{N\times N}|\text{ } A=A^{T} \}\nonumber
\end{equation}
 For a matrix $A\in S\R^{N\times N}$, the notation $A\succeq 0$ means that $A$ is positive semidefinite, i.e. $a^{T}Aa\geq 0$ for all $a\in\R^{N}$. Let $v_{1},\ldots,v_{r}\in\R^{N}$ we write $\spn\left\langle v_{1},\ldots,v_{r}\right\rangle$ to refer us to the real linear subspace generated by the vectors $v_{1},\ldots,v_{r}$.


\section{Introduction and Preliminaries}

Let $f\in\R[\underline{X}]$. Let us consider a polynomial optimization problem without constraints, that is to say we consider the problem of find the minimum if possible, and minimizers if possible of the following polynomial optimization problem:
\begin{equation}\label{Popw}
\begin{aligned}
(P) \text{  minimize  } f(x) \text{ subject to } x\in \R^{n}
\end{aligned}
\end{equation}
The optimal value of $(P)$, i.e. the infimum of $f(x)$ where $x$ ranges over $\R^{n}$ will be denoted by $P^{*}$, that is to say:
\begin{equation}
\begin{aligned}
P^{*}:= 
 \inf\{\text{ }f(x)\text{ } |\text{ } x\in \R^{n}  \}\in\{-\infty\}\cup\R\cup\{\infty\}\nonumber
\end{aligned}
\end{equation}
In this paper we present an heuristic algorithm to find, in some cases, an upper bound $U$ of $P^{*}$, that is to say $U\geq P^{*}$ and if possible points $a\in\R^{n}$ such that $f(a)=U$. Let us first recall some preliminaries of basic concepts in semidefinite optimization that we will use in the final algorithm.

\begin{dyn} For $n,m\in\N_{0}$ a semidefinite program (in primal form) is a program of the following form:
\begin{align}\label{psd}
 (SDP)_{\ell,L}\text{   minimize }
\ell(x) &\text{ subject to: } \\ \nonumber
&
x \in\R^{n} \text{ and} \ \ L(x)\succeq 0  \\ \nonumber
&
\text{where } \ell\in\R[\underline{X}]_{1}\text{ and } L\in S\R[\underline{X}]_{1}^{m\times m}\text{ are given. }\nonumber
\end{align}
The optimal value of $(SDP)_{\ell,L}$, that is to say the infimum over all $x\in\R^{n}$ that ranges over all feasible solutions of $(SDP)_{\ell,L}$ is denoted by $(SDP)_{\ell,L}^{*}\in\{-\infty \}\cup\R\cup\{+\infty \}$.
\end{dyn}

\begin{rem}
Note that for $p\in\R[\underline{X}]_{1}$ we can add the linear condition $p(x)=0$ for all $x\in\R^{n}$ to the positive semidefinite program $(SDP)_{\ell,L}$ \eqref{psd}, by adding $p(x)\geq 0$ and $-p(x)\geq 0$ to the diagonal of a bigger symmetric matrix, that is to say considering:
\begin{equation}
L(x) := \left(
\begin{array}{r|r|r}
 L(x) & 0 & 0 \\ \cline{1-3}
 0 &-p(x) & 0 \\ \cline{1-3}
 0 & 0 & p(x)
   \end{array}
    \right)\succeq 0 \nonumber
\end{equation}
\end{rem}

Summarizing a semidefinite program is the cone of the positive semidefinite matrices intersected with a linear subspace. Semidefinite programs can also be seen as generalization of linear programs since a linear program is a semidefinite program $(SDP)_{\ell,L}$ where $L$ is a diagonal matrix. Semidefinite programs are possible to solve efficiently and there are many softwares and packages that allows to solve them, in particular we will use SEDUMI and YALMIP, see \ref{software}.

Let us recall how to try to solve the polynomial optimization problem $(P)$ \eqref{Popw}, by
solving a hierarchy of very well known semidefinite programs called Moment Relaxation or Lasserre relaxation of certain degree. For this, let $d\in\N_{0}$ and let us define: 
\begin{align}
    V_{d}:= (&1,X_{1}, X_{2}, \ldots,  X_{n},X_{1}^2,X_{1}X_{2},\ldots,X_{1}X_{n},\\
		&X_{2}^2,X_{2}X_{3},\ldots,X_{n}^2,\ldots,\ldots,X_{n}^d)^{T} \nonumber
		\end{align}
as a basis for the vector space of polynomials in $n$ variables of degree at most $d$. Then
\begin{center}
$V_{d}V_{d}^{T}=\left(
\begin{array}{ccccc}
 1         & X_{1}         & X_{2}          &  \cdots &    X_{n}^{d}    \\
 X_{1}     & X_{1}^2       & X_{1}X_{2}     &  \cdots &    X_{1}X_{n}^{d}   \\
 X_{2}     & X_{1}X_{2}    & X_{2}^2        &  \cdots &    X_{2}X_{n}^{d} \\
 \vdots    & \vdots        & \vdots         &  \ddots &    \vdots \\
 X_{n}^{d} &X_{1}X_{n}^{d} & X_{2}X_{n}^{d} &  \cdots &    X_{n}^{2d}\\
    \end{array}
   \right)\in S\R[\underline{X}]_{2d}^{s_{d}\times s_{d}}$
	\end{center}

Let us substitute for every monomial $\underline{X}^{\alpha}\in\R[\underline{X}]_{2d}$ a new variable $Y_{\alpha}$. This matrix has the following form: 
\begin{equation}\label{generalizedHankel}
M_{d}:=\left(
\begin{array}{ccccc}
 Y_{(0,\ldots,0)}          & Y_{(1,\ldots,0)}               & Y_{(0,1,\ldots,0)}          &  \cdots &    Y_{(0,\ldots,1)}    \\
 Y_{(1,\ldots,0)}         & Y_{(2,\ldots,0)}              & Y_{(1,1,\ldots,0)}          &  \cdots &     Y_{(1,\ldots,d)}   \\
 Y_{(0,1,\ldots,0)}       & Y_{(1,1,\ldots,0)}            &  Y_{(0,2,\ldots,0)}          &  \cdots &   Y_{(0,1,\ldots,d)}   \\
 \vdots                    & \vdots        & \vdots         &  \ddots &    \vdots \\
Y_{(0,\ldots,d)}       & Y_{(1,\ldots,d)}     & Y_{(0,1,\ldots,d)}    &  \cdots &   Y_{(0,\ldots,2d)}   \\
    \end{array}
   \right)\in S\R[\underline{Y}]_{1}^{s_{d}\times s_{d}}
\end{equation}

\begin{dyn} Every matrix $M\in\R^{s_{d}\times s_{d}}$ with the same shape than the matrix \eqref{generalizedHankel} is called a generalized Hankel matrix (or Moment matrix) of order $d$.
\end{dyn}
Let us shortly explain how the Lasserre relaxation transform the polynomial optimization problem $(P)$ \eqref{Popw} into a semidefinite program \eqref{psd}. The idea is that, for every $d\in\N_{0}$ the problem $(P)$ is equivalent to the following problem:
\begin{align}\label{explanationauxiliar}
 \text{   minimize } \nonumber  f(x)  &\text{ subject to }\\  
& x \in\R^{n} \text{ and } p(x)^{2}\geq 0 \text{ for all }p\in\R[\underline{X}]_{d}  
\end{align}
Let us denote $\tilde{p}:=(p_{(0,\ldots,0)},\ldots,p_{(0,\ldots,d)})^{T}\in\R^{s_{2d}}$, the vector with the coefficients of $p$. Then the trivial inequalitiy $p(x)^{2}\geq 0$ for all $x\in\R^{n}$ and for all $p\in\R[\underline{X}]_{d}$ can be written as $\tilde{p}^{T}V_{d}V_{d}^{T}(x)\tilde{p}\geq 0$ for all $x\in\R^{n}$ and for all $\tilde{p}\in\R^{s_{2d}}$, and this last equality can also be writen as $V_{d}V_{d}^{T}(x)\succeq 0$ for all $x\in\R^{n}$.

Since $V_{d}V_{d}^{T}\in S\R[\underline{X}]_{2d}^{s_{d}\times s_{d}}$ is not a matrix with linear entries, the next idea is to substitute every monomial  $\underline{X}^{\alpha}$  for a new variable $Y_{\alpha}$ in this way we will not have anymore an equivalent problem to \eqref{Popw} but a "relaxation" of  the problem, that is to say the feasible set will be bigger and consequently by solving this relaxation problem we will get a lower bound of the infimum. For better introduction of the moment relaxation with more details we refer the reader to \cite{lau}, \cite{mar}, \cite{sch} and references therein.
 
\begin{dyn}\label{relax}
Let $(P)$ be a polynomial optimization problem  as in \eqref{Popw} and 
let $k\in$ $\N_{0}\cup\{\infty\}$ such that $f\in\R[\underline{X}]_{k}$. The Moment relaxation (or Lasserre relaxation) of $(P)$ of degree $k$ is the following semidefinite optimization problem:
\begin{align}
(P_{k}) \text{ minimize } \sum_{|\alpha|\leq k} f_{\alpha}y_{\alpha}\nonumber & \text{ subject to } \\ \nonumber
& M_{\lfloor\frac{\deg{k}}{2}\rfloor}(y)\succeq 0 \text{ and }  y_{\left(0,\ldots,0\right)}=1 \nonumber
\end{align}
the optimal value of $(P_{k})$ that is to say, the infimum over all
\begin{equation}
 y=(y_{\left(0,\ldots,0 \right)},\ldots,y_{\left(0,\ldots,k \right)})\in\R^{s_{k}}\nonumber
 \end{equation} that ranges over all feasible solutions of $(P_{k})$ is denoted by $P^{*}_{k}\in\{-\infty\}\cup\R\cup\{\infty\} $.
\end{dyn}

Let us remember some trivial properties of the Moment relaxations:

\begin{prop}\label{motivation}
Let $(P)$ be a polynomial optimization problem as in $\eqref{Popw}$ and let $k\in\N_{0}\cup\{ \infty \}$ such that $f\in\R[\underline{X}]_{k}$. Set  $d:=\lfloor\frac{\deg{k}}{2}\rfloor$. The following holds:
\begin{enumerate}
\item $P^{*}\geq\ldots\geq P^{*}_{k+1}\geq P^{*}_{k}$
\item  Every matrix of the form:
\begin{equation}
 M=\sum_{i=1}^{N}\lambda_{i}V_{d}V_{d}^{T}(a_{i})\in S\R^{s_{d} \times s_{d}}\nonumber
 \end{equation}
  with $a_{i}\in\R^{n}$ for all $i\in\{1,\ldots,N\}$ and with $\sum_{i=1}^{N}\lambda_{i}=1$ is the Moment matrix of a feasible solution $y\in\R^{s_{k}}$ of $(P_{k})$, i.e. $M=M_{d}(y)$, and $\sum_{|\alpha|\leq k} f_{\alpha}y_{\alpha}\geq P^{*}$.
  \item If $(P_{k})$ has an optimal solution $y\in\R^{s_{k}}$ such that there exists $N\in\N$ and $a_{1},\ldots,a_{N}\in\R^{n}$ and $\lambda_{1}>0,\ldots,\lambda_{N}>0$ such that:
	\begin{equation}
	M_{d}(y)=\sum_{i=1}^{N}\lambda_{i}V_{d}V_{d}^{T}(a_{i})\in S\R^{s_{d} \times s_{d}}\nonumber
	\end{equation}
	Then $P^{*}=P^{*}_{k}$ and $a_{1},\ldots,a_{N}$ are minimizers of $f$.
\end{enumerate}

\end{prop}
\begin{proof}
For a proof of this Proposition we refer to \cite[Proposition 3.9]{mlq1}.
\end{proof}

Let us now, recall the Theorem \ref{above} that will be use in the Algorithm \ref{upper}. The Theorem \ref{above} gives us a condition to detect optimality in an optimal solution of the Moment relaxation, that is to say $P^{*}=P_{k}$.


\begin{prop}
Let $y\in\R^{s_{k}}$ be a feasible solution of $(P_{k})$, set $d:=\left\lfloor{\frac{k}{2}}\right\rfloor$ and $M:=M_{d}(y)$. There exist $W\in\R^{s_{d} \times r_{d}}$ and $C\in\R^{r_{d}\times r_{d}}$ such that $M$ can be decomposed in a block matrix of the following form:
$$M=
\left(
\begin{array}{c|c}
\makebox{$A$}&\makebox{$AW$}\\
\hline
  \vphantom{\usebox{0}}\makebox{$W^{T}A$}&\makebox{$C$}
\end{array}
\right)
$$
 For the matrix $M$ we define and denote its respective modified Moment matrix as it follows:
\begin{equation}
\widetilde{M}:=\left(
\begin{array}{c|c}
\makebox{$A$}&\makebox{$AW$}\\
\hline
  \vphantom{\usebox{0}}\makebox{$W^{T}A$}&\makebox{$W^{T}AW$}
\end{array}\right)\nonumber
\end{equation}
and $\widetilde{M}$ is well defined that is to say, it does not depend from the choice of $W$.
\end{prop}
\begin{proof}
This useful result can be also found in \cite{Smu} and in \cite[Lemma 4.8]{mlq1}.
\end{proof}

\begin{defi}
Let $y\in\R^{s_{k}}$ be a feasible solution of $(P_{k})$ and set $d:=\left\lfloor{\frac{k}{2}}\right\rfloor$. We say $M:=M_{d}(y)$ is a flat matrix if the following condition holds:
\begin{equation}
M=\widetilde{M}
\end{equation}
\end{defi}

\begin{rem}
Let $y\in\R^{s_{k}}$ be a feasible solution of $(P_{k})$ and set $d:=\left\lfloor{\frac{k}{2}}\right\rfloor$. Note that $M:=M_{d}(y)$ is a flat matrix if the following condition in the rank of $M$ holds:
\begin{equation}
\rk M_{d}(y)= \rk M_{d-1}(y)
\end{equation}
\end{rem}

\begin{teor}\label{above}
Let $f\in\R[\underline{X}]_{k}$ and $(P)$ be the polynomial optimization problem without constraints defined in \eqref{Popw}, and let $y\in\R^{s_{k}}$ be an optimal solution of the Moment relaxation $(P_{k})$ and set $d:=\left\lfloor{\frac{k}{2}}\right\rfloor$. Then the following conditions hold:
\begin{enumerate}
\item If $\widetilde{M_{d}(y)}$ is generalized Hankel and $f\in\R[\underline{X}]_{k-1}$ then $P^{*}=P^{*}_{k}$, there exit $\lambda_{1}>0,\ldots,\lambda_{r}>0$ and $a_{1},\ldots,a_{r}\in\R^{n}$ such that:
\begin{equation}
\widetilde{M_{d}(y)}=\sum_{i=1}^{r}\lambda_{i}V_{d}V_{d}^{T}(a_{i})\in S\R^{s_{d} \times s_{d}}\nonumber
\end{equation}
and $a_{1},\ldots,a_{r}$ are minimizers of $f$.
\item If $M_{d}(y)$ is flat then $P^{*}=P^{*}_{k}$, there exit $\lambda_{1}>0,\ldots,\lambda_{r}>0$ and $a_{1},\ldots,a_{r}\in\R^{n}$ such that:
\begin{equation}
M_{d}(y)=\sum_{i=1}^{r}\lambda_{i}V_{d}V_{d}^{T}(a_{i})\in S\R^{s_{d} \times s_{d}}\nonumber
\end{equation}
and $a_{1},\ldots,a_{r}$ are minimizers of $f$.
\end{enumerate} 
\end{teor}
\begin{proof}
The proof is in  \cite[Corollary 7.3]{mlq1}.
\end{proof}

\section{Main ideas in the Algorithm} 
Let $(P)$ the polynomial optimization problem defined in \eqref{Popw} and $f\in\R[\underline{X}]_{k}$. Given $y\in\R^{s_{k}}$ an optimal solution of $(P_{k})$, set $d:=\left\lfloor{\frac{k}{2}}\right\rfloor$ and $M:=M_{d}(y)$. It is not always the case that $M$ is flat or it is not always the more general case that $\widetilde{M}$ is a generalized Hankel matrix, in this case, in order to find the minimum $P^{*}$ and minimizers, we could try to increase $k$ and solve again the Moment relaxation and hope that we get an optimal solution with $M$ flat or $\widetilde{M}$ generalized Hankel. However the dimension of the problem could increase considerably and one frequently runs into numerical problems. Therefore in the Algorithm \ref{upper} we try to modify a little bit the optimal solution $y$ to get a flat solution or a solution close to be flat, this way we try to avoid to increase $k$. A first try to get a flat optimal solution of $(P_{k})$ would be to add linear constraints into the Moment relaxation in order to restrict our set of feasible solutions to a set of flat feasible solutions or at least "close" to be flat. Let me explain shortly why this is in principle, a hard problem. As we have mentioned before, a first approach would be to try to describe the following program:
\begin{align}
& \text{ minimize }\nonumber
\sum_{|\alpha|\leq k} f_{\alpha}y_{\alpha} \text{ subject to: } \\ \nonumber
&
M_{d}(y)\succeq 0, \ \ y_{\left(0,\ldots,0\right)}=1, \ \ \rk M_{d}(y)=\rk M_{d-1}(y) \nonumber
\end{align}
as a positive semidefinite program but $\rk M_{d}(y)=\rk M_{d-1}(y)$ if and only if $M_{i}\in \spn\left\langle M_{1},\ldots,M_{s_{d-1}} \right\rangle$ for all $i\in\{ s_{d-1}+1,\ldots,s_{d}\}$. However we can not add the constraints:
\begin{align}\label{condition}
M_{i}= \sum_{j=1}^{s_{d-1}}a_{j}^{i}M_{j} & \text{ for some }a_{1}^{i},\ldots,a_{s_{d-1}}^{i}\in\R \\ 
 & \text{ for all }i\in\{s_{d-1}+1,\ldots, s_{d}\}\nonumber
\end{align} 

to our Moment relaxation since this condition is not linear due to the $a_{i}$ and the entries of the matrix are decision variables or unknows, and this can not be written, at least not in any obvious way, as a semidefinite program. In the same way, the program:
\begin{align}
& \text{ minimize } \nonumber
\sum_{|\alpha|\leq k} f_{\alpha}y_{\alpha} \text{ subject to: } \\ \nonumber
&
M_{d}(y)\succeq 0 \text{, }y_{\left(0,\ldots,0\right)}=1, \text{ and }\widetilde{M_{d}(y)}\text{ is generalized Hankel} \nonumber
\end{align}
is not a positive semidefinite program due to that $W_{M_{d}(y)}^{T}M_{d-1}(y)W_{M_{d}(y)}$ is not a linear matrix since the entries of $M_{d-1}(y)$ and the entries of $W_{M_{d}(y)}$ are decision variables. Moreover to solve  polynomial optimization problems without constraints already for degree $4$ polynomials is NP hard \cite{NP}, so it is reasonable to expect that to convert these programs into a  semidefinite program is hard. Nevertheless, we can modify a little bit the optimal solution $y$ into $y_{0}\in\R^{s_{k}}$ in such a way that $y_{0}$ is feasible solution of $(P_{k})$ and $M_{d}(y_{0})$ is approximately flat. Since $y_{0}$ is a feasible solution of $(P_{k})$ the inequality $\sum_{|\alpha|\leq k }f_{\alpha}(y_{0})_{\alpha}\geq P_{k}^{*}$ holds. Moreover if $M_{d}(y_{0})$ is flat then by \ref{motivation} $(2)$ we know that $\sum_{|\alpha|\leq 2d }f_{\alpha}(y_{0})_{\alpha}\geq P^{*}$. More precisely in this last case it holds that:
\begin{equation}
 P^{*}\in [P^{*}_{k},\sum_{|\alpha|\leq k }f_{\alpha}(y_{0})_{\alpha}]
\end{equation}

\begin{Reminder}\label{nube}
Let us consider the following polynomial optimization problem, called the Least Squares Problem:
\begin{equation}\label{Pop}
\begin{aligned}
(P_{A,b}) \text{  minimize  } ||Ax-b||^{2}_{2}, \text{ subject to } x\in \R^{n}
\end{aligned}
\end{equation}
where $A\in\R^{m\times n}$, $b\in\R^{m}$ are given and $m>n$. Minimizers of this problem are called a least squares approximate solutions. Suppose the matrix $A^{T}A$ is non singular then the unique solution, denoted by $x_{A,b}^{*}$,  of the least squares problem is given by:
\begin{equation}\label{vectorprojection}
x_{A,b}^{*}=(A^{T}A)^{-1}A^{T}b
\end{equation}
\end{Reminder}

For a proof of the Reminder \ref{nube} and more details about the topic we refer the reader to \cite{BoyVan} and references therein.

Given a polynomial optimization without constraints $(P)$ $\eqref{Popw}$, with $f\in\R[\underline{X}]_{k}$ and an optimal solution $y\in\R^{s_{k}}$ of the Moment relaxation $(P_{k})$, the next step in the algorithm is to build the closest matrix to $M:=M_{d}(y)$, let us denoted it by $B_{M}\in\R^{s_{d}\times s_{d}}$, with:
\begin{equation}
\rk(B_{M_{1}}\cdots B_{M_{s_{d}}})=\rk(B_{M_{1}}\cdots B_{M_{s_{d-1}}})\nonumber
\end{equation}
that is to say the first $s_{d-1}$ columns of $B_{M}$ are the same as the first $s_{d-1}$ columns of $M$, the last $r_{d}$ columns of $B_{M}$ belong to the real linear span of this columns and $B_{M}$ is the closed matrix to $M$ in the sense that the column $B_{M_{j}}$ for all $j\in\{s_{d-1}+1,\ldots,s_{d}\}$ is the closest vector to $M_{j}$ which lies in the real linear span $\left\langle B_{M_{1}},\ldots,B_{M_{s_{k}}}\right\rangle$, that is to say $B_{M_{j}}$ is the orthogonal projection of $M_{j}$ into  $\left\langle B_{M_{1}},\ldots,B_{M_{s_{k}}}\right\rangle$. Then  $B_{M_{j}}=M_{d-1}(y)x^{*}_{M_{d-1}(y),M_{j}}$ for all $j\in\{s_{d-1}+1,\ldots,s_{d}\}$ where $x^{*}_{M_{d-1}(y),M_{j}}$ is the least squares approximate solution of $(P_{M_{d-1}(y),M_{j}})$.\\

The matrix $B_{M}$ holds the desired condition in the rank \eqref{condition}, however it is not necessarily positive semidefinite, not even symmetric and also not generalized Hankel, that is to say is not a feasible solution of $(P_{k})$. So now we look for $y_{0}\in\R^{s_{k}}$, such that $E\in\R$ is the smallest possible in the following inequality:
\begin{equation}\label{amy}
||M_{d}(y_{0})-B_{M}||^{2}\leq E||M-B_{M}||^{2}             
\end{equation}
Setting $A_{M}:=M_{d}(y_{0})$, we will solve the following program:
\begin{align}
(P_{M}) \text{   minimize } \nonumber E  & \text{ subject to }\\  \nonumber
& E \in\R \text{ and } ||A_{M}-B_{M}||^{2}\leq E||M-B_{M}||^{2}  \nonumber
\end{align}
Note that in the program $(P_{M})$ the decision variables $y_{0}\in\R^{s_{k}}$ and $E\in\R$.  With the condition \eqref{amy} we attempt to simultaneously control the rank of $A_{M}$ by minimizing the distance from $A_{M}$ to $B_{M}$ and at the same time we get a matrix with lower or equal rank than the original matrix $M$, since note the inequality \eqref{amy} holds taking $E:=1$ and $A_{M}:=M$, there exists always a feasible solution.  The Schur complement, defined in \ref{Schur}, will enable us to show that this condition is equivalent to the positive semidefiniteness of a matrix with linear entries and therefore we can conclude that $(P_{M})$ is a positive semidefinite program, possible to solve efficiently.

\begin{defi}\label{Schur}
Let us consider a  matrix $X\in S\R^{m+l \times m+l}$ in block form:
\begin{equation}\label{schur}
X=\left(
\begin{array}{r|r}

A & B \\ \cline{1-2}
B^{T} & C  
\end{array}
\right)
\end{equation}
with $A\in\R^{m\times m}$, $B\in\R^{m\times l}$, $C\in\R^{l \times l}$. Assume $A$ is non-singular. Then the matrix $C-B^{T}A^{-1}B$ is called the Schur complement of $A$ in $X$.
\end{defi}

\begin{lem}\label{algoritmo2}
Let $X\in\mathcal S^{m\times m}$ be in block form \eqref{schur}, where $A$ is non-singular. Then,
\begin{equation}
X\succeq 0 \iff A\succeq 0\text{ and }C-B^{T}A^{-1}B\succeq 0\nonumber
\end{equation}
\end{lem}
\begin{proof}
For a proof of this lemma we refer the reader to \cite[Lemma 1.7.6]{MoVa}
\end{proof}
%

Therefore applying Lemma \ref{algoritmo2} we got that:

\begin{equation}
||A_{M}-B_{M}||^{2}\leq E||M-B_{M}||^{2}\iff\nonumber
\end{equation}

%

\begin{equation}\label{linear}
\left(
\begin{array}{c|c}
I_{s_{d}s_{d}}& \begin{array}{c} (A_{M})_{1,1}-(B_{M})_{1,1}\\ \vdots\\  (A_{M})_{s_{d},s_{d}}-(B_{M})_{s_{d},s_{d}} \end{array} \\ \cline{1-2}

\begin{array}{ccc}(A_{M})_{1,1}-(B_{M})_{1,1} & \cdots  & (A_{M})_{s_{d},s_{d}}-(B_{M})_{s_{d},s_{d}} \end{array} & E||M-B_{M}||^{2}
   \end{array}
    \right)\succeq 0
\end{equation}

The matrix \eqref{linear} is linear in the variables $E$ and the entries of the matrix $A_{M}$, that is in $y_{0}\in\R^{s_{k}}$. We could solve directly the semidefinite program $(P_{M})$ to get a lower bound of $P^{*}$. However in practice if instead we consider the following semidefinite program for $\lambda\in[0,1]$ fixed, we get better bounds:
\begin{align}
(P_{M,\lambda}) \text{   minimize } \nonumber E\lambda+(1-\lambda)\sum_{|\alpha|\leq k}f_{\alpha}y_{0,\alpha}  & \text{ subject to }\\  \nonumber
& E \in\R \text{ , } ||M_{k}(y_{0})-B||^{2}\leq E||M-B_{M}||^{2}\\  \nonumber
& M_{k}(y_{0})\succeq 0 
\end{align}
The optimal value of $(P_{M,\lambda})$ that is to say, the infimum over all:
\begin{equation}
(y_{0,\left(0,\ldots,0 \right)},\ldots,y_{0,\left(0,\ldots,k \right)})\in\R^{s_{k}}\text{ and }E\in\R\nonumber
\end{equation}
that ranges over all optimal solutions of $(P_{M,\lambda})$ is denoted by $P^{*}_{M,\lambda}$.

 \section{Algorithm 1 based on Moment relaxations}
\begin{algorithm}\label{upper}
  \KwIn{A polynomial optimization problem $(\textbf{P})$ \eqref{Pop} without constraints and an
	strategic $\lambda\in [0,1] $}
  \KwOut{An upper bound $U\geq\textbf{P}^{*}$ and if possible, points $\textbf{a}_{1},\ldots,\textbf{a}_{r}$ such that $f(\textbf{a}_{\textbf{i}})=U$.}\
  Compute a feasible solution of the Moment relaxation $(P_{\deg f})$ and denote it by $y\in\R^{s_{\deg f}}$. Set  $d:=\lfloor\frac{\deg{f}}{2}\rfloor$  and $\textbf{M}:=M_{d}(y)$. \;
Take a maximum linearly independent subset of the first $s_{d-1}$ columns of $\textbf{M}$ and denote it by $\{M_{1},\ldots,M_{l}\}$ and set $C:=(M_{1}\cdots M_{l})\in \R^{s_{d}\times l}$. \;
Define the folowing matrix using the definition of the solution of the least squares problem \eqref{vectorprojection},
$\textbf{B}_{\textbf{M}}:=(M_{1}\cdots M_{s_{d-1}}|Cx^{*}_{C,M_{s_{d-1}+1}}\cdots Cx^{*}_{C,M_{s_{d}}})$.\;
Finally define the following matrix:
\begin{equation}
\textbf{T}_{\textbf{M}}:=\left(
\begin{array}{c|c}
I_{s_{d}s_{d}}& \begin{array}{c} (A_{M})_{1,1}-(B_{M})_{1,1}\\ \vdots\\  (A_{M})_{s_{d},s_{d}}-(B_{M})_{s_{d},s_{d}} \end{array} \\ \cline{1-2}

\begin{array}{ccc}(A_{M})_{1,1}-(B_{M})_{1,1} & \cdots  & (A_{M})_{s_{d},s_{d}}-(B_{M})_{s_{d},s_{d}} \end{array} & E||M-B_{M}||^{2}
   \end{array}
    \right)\nonumber
\end{equation}
where the decision variables are $\textbf{E}$ and the entries of the matrix $\textbf{A}_{\textbf{M}}\in S\R^{s_{d}\times s_{d}}$. \;
Solve the following positive semidefinite program:
\begin{align}
(P_{M,\lambda}) \text{   minimize } \nonumber E\lambda+(1-\lambda)\sum_{|\alpha|\leq k}f_{\alpha}y_{0,\alpha}  & \text{ subject to }\\  \nonumber
& \textbf{A}_{\textbf{M}}=M_{d}(y_{0})\\  \nonumber
& E \in\R \text{ , } \textbf{T}_{\textbf{M}}\succeq 0\\  \nonumber
& A_{M}\succeq 0 
\end{align}
Note that here the decision variables are $E$ and the entries of the matrix $\textbf{A}_{\textbf{M}}\in S\R^{s_{d}\times s_{d}}$ which are $y_{0}\in\R^{s_{k}}$. \;
\If {$\textbf{A}_{\textbf{M}}$ is flat or $||\textbf{A}_{\textbf{M}}-\textbf{B}_{\textbf{M}} ||$ is small enough} { $\textbf{U}:=\sum_{|\alpha|\leq k}f_{\alpha}y_{0,\alpha}$. If possible extract $a_{1},\ldots a_{r}\in\R^{n}$ such that $f(a_{i})=\textbf{U}$ with \cite[Algorithm 1]{mlq1}.}
\Else{$\textbf{M}:=\textbf{A}_{\textbf{M}}$ and go to $\textbf{3}$} 
  \caption{Given $(P)$ \eqref{Popw}, finding an upper bound $U$ of $P^{*}$, i.e $U\geq P^{*}$ and if possible minimizers or potential minimizers  }
  \end{algorithm}
	
	\begin{rem}
We do not know if the algorithm ever terminates. In the examples we did, if the algorithm took too much time, we interrupted the algorithm even if the matrix $\textbf{A}_{\textbf{M}}$ was not flat.	Note also that since $y_{0}$ is a feasible solution of $P_{k}$, $U\geq P^{*}_{k}$ and in case $\textbf{A}_{\textbf{M}}$ is flat we can even conclude by \ref{motivation} $(2)$ that:
	\begin{equation}
	P^{*}\in [P^{*}_{k}, \textbf{U}]\nonumber
	\end{equation}
	\end{rem}

\section{Numerical Results of Algorithm 1}
\begin{ej} The \textit{Motzkin polynomial}, $X_{1}^{4}X_{2}^{2}+X_{1}^2X_{2}^{4}-3X_{1}^2X_{2}^2+1$, is nonnegative but not sum of squares (see \cite[Proposition 1.2.2]{mar} for a proof). Let us consider the following polynomial optimization problem:
\begin{equation*}
\begin{aligned}
& {\text{minimize}}
& & f(x)=x_{1}^{4}x_{2}^{2}+x_{1}^2x_{2}^{4}-3x_{1}^2x_{2}^2+1 \\
& \text{subject to}
& & x\in\R^{2}\\
\end{aligned}
\end{equation*}
Since we know that the Motzkin polynomial is nonnegative and $f(\pm 1,\pm 1)=0$ is not difficult to see that $P^{*}=0$. Let us use this polynomial to see how the Algorithm \ref{upper} works. An optimal solution $z\in\R^{s_{6}}$ of the Moment relaxation $(P_{6})$, has the following Moment matrix:
{\scriptsize
\begin{align} \nonumber
M_{3}(z)= \begin{blockarray}{ccccccc}
\text{ } & 1 & X_{1} & X_{2} & X_{1}^2 & X_{1}X_{2} & X_{2}^2 \\
\begin{block}{c(cccccc}
1&1.00&0.00&0.00&5300.32&0.00&5300.31\\
X_{1} &0.00&5300.32&0.00&0.00&0.00&0.00\\
X_{2} &0.00&0.00&5300.31&0.00&0.00&0.00\\
X_{1}^{2}&5300.32&0.00&0.00&57120303.73&0.00&2966.08\\
X_{1}X_{2}&0.00&0.00&0.00&0.00&2966.08&0.00\\
X_{2}^{2}&5300.31&0.00&0.00&2966.08&0.00&57120215.99\\
X_{1}^{3}&0.00&57120303.73&0.00&2.14&0.00&0.00\\
X_{1}^{2}X_{2}&0.00&0.00&2966.08&0.00&0.00&0.00\\
X_{1}x_{2}^{2}&0.00&2966.08&0.00&0.00&0.00&0.00&\\
X_{2}{3}&0.00&0.00&57120215.99&0.00&0.00&-24.83\\
\end{block}
   \end{blockarray}\\ \nonumber
\begin{blockarray}{ccccccc}
X_{1}^{3} & X_{1}^{2}X_{2} & X_{1}X_{2}^{2} & X_{2}^{3} & \text{ }\\
\begin{block}{cccc)ccc}
0.00&0.00&0.00&0.00&1\\
57120303.73&0.00&2966.08&0.00&X_{1}\\
0.00&2966.08&0.00&57120215.99&X_{2}\\
-2.14&0.00&0.00&0.00&X_{1}^{2}\\
0.00&0.00&0.00&0.00&X_{1}X_{2}\\
0.00&0.00&0.00&-24.83&X_{2}^{2}\\
2272816741819.88&0.00&2491.80&0.00&X_{1}^{3}\\
0.00&2491.80&0.00&2491.80&X_{1}^{2}X_{2}\\
2491.80&0.00&2491.80&0.00&X_{1}X_{2}^{2}\\
0.00&2491.80&0.00&2272810890195.67&X_{2}^{3}\\
\end{block}
   \end{blockarray}\\ \nonumber
\end{align}}
This matrix is neither flat nor $\widetilde{M_{3}(z)}$ is a generalized Hankel matrix so we can not conclude optimality \ref{above}, in fact the optimal value is $P^{*}_{6}=-\infty<<P^{*}=0$ stricly smaller than the minimum. Since the entries of the matrix $M_{3}(z)$ are very far to each other and working with this matrix could bring us numerical problems, so instead we take a random feasible solution of $(P_{6})$, $y\in\R^{s_{6}}$ with moment matrix $M:=M_{3}(y)$. Compute, as in the Algorithm \ref{upper}, the matrices $B_{M}$ and $T_{M}$ and solve the semidefinite program $(P_{M,1/60})$. We get after $66$ iterations the optimal solution $y_{0}\in\R^{s_{6}}$  with Moment matrix:

{\scriptsize\begin{align}A_{M}:=M_{3}(y_{0})\nonumber= \begin{blockarray}{ccccccccccc}
\text{ } & 1 & X_{1} & X_{2} & X_{1}^2 & X_{1}X_{2} & X_{2}^2 & X_{1}^{3} & X_{1}^{2}X_{2} & X_{1}X_{2}^{3} & X_{2}^{3} \\
\begin{block}{c(cccccccccc)}
1 &1.00&0.00&0.00&1.02&0.00&1.02&0.00&0.00&0.00&0.00\\
X_{1}&0.00&1.02&0.00&0.00&0.00&0.00&1.04&0.00&1.04&0.00\\
X_{2}&0.00&0.00&1.02&0.00&0.00&0.00&0.00&1.04&0.00&1.04\\
X_{1}^{2}&1.02&0.00&0.00&1.04&0.00&1.04&0.00&0.00&0.00&0.00\\
X_{1}X_{2}&0.00&0.00&0.00&0.00&1.04&0.00&0.00&0.00&0.00&0.00\\
X_{2}^{2}&1.02&0.00&0.00&1.04&0.00&1.04&0.00&0.00&0.00&0.00\\
X_{1}^{3}&0.00&1.04&0.00&0.00&0.00&0.00&1.07&0.00&1.07&0.00\\
X_{1}^{2}X_{2}&0.00&0.00&1.04&0.00&0.00&0.00&0.00&1.07&0.00&1.07\\
X_{1}X_{2}^{2}&0.00&1.04&0.00&0.00&0.00&0.00&1.07&0.00&1.07&0.00\\
X_{2}^{3}&0.00&0.00&1.04&0.00&0.00&0.00&0.00&1.07&0.00&1.07\\
\end{block}
   \end{blockarray}\\ \nonumber
	\end{align}}

Moreover we get that $U:=\sum_{|\alpha|\leq 6}f_{\alpha}y_{0,\alpha}=0.00156$. It turns out that this matrix is flat and therefore is an upper bound of the minimum $P^{*}\leq U$. Moreover we can extract the points $a_{i}$ such that $f(a_{i})=U$ with the method \cite[Algorithm 1]{mlq1}. Applying this algorithm we get the factorization:
\begin{align}
A_{M}=\nonumber&\frac{1}{4}\cdot V_{3}V_{3}^{T}(-1.0109,-1.0109)+\frac{1}{4}\cdot V_{3}V_{3}^{T}(1.0109,-1.0109)+\\
&\frac{1}{4}\cdot V_{3}V_{3}^{T}(-1.0109,1.0109)+\frac{1}{4}\cdot V_{3}V_{3}^{T}(1.0109,1.0109)\nonumber
\end{align}
Therefore $f(\pm 1.0109,\pm 1.0109)=U$, and this four points are close to the minimizers of the Motzkin polynomial that are $(\pm 1, \pm 1)$.
The table \ref{table:1} shows the different data that we get by solving the semidefinite program $(P_{M,\lambda})$ for different values of $\lambda\in [0,1]$. Note that if the algorithm \ref{upper} did  not seem to reach a flat $A_{M}$ then we manually stopped the algorithm when took more than $3$ minutes. Note that for $\lambda=\frac{3}{4}$ we get an approximately flat solution and we extract potential minimizers by block-simultaneous diagonalization the truncated multiplication operators defined in \cite[Definition 4.5]{mlq1} and taking the blocks of dimension $1$. The algorithm for the block simultaneous diagonalization is described in \cite[Algorithm 4.1]{japoneses}.

\begin{table}
\begin{tabular}{ |c|c|c|c|c|c|c| } 
\hline
 & $U$ & Flat & $||A_{M}-B_{M}||$ & Potential Minimizers & Iterations & Time \\ \cline{1-7}
$\lambda=\frac{1}{1000}$ & $-209.9332$& No& $1.8995\cdot 10^{7}$&-&$73$&$2^{M}14^{S}$\\ \cline{1-7}
 $\lambda=\frac{1}{100}$ & $-48.50$  & No & $1.584\cdot 10^{5}$ & - & $73$ &  $1^{M}42^{S}$ \\ \cline{1-7}
 $\lambda=\frac{1}{60}$ & $0.00156$ & Yes & $9.16135$ & $(\pm 1.0109,\pm 1.0109)$ & $66$ &  $1^{M}4^{S}$ \\ \cline{1-7}
 $\lambda=\frac{1}{4}$ & $0.0650$ & No & $2.0449\cdot 10^{-4}$ & - & $66$ &  $1^{M}15^{S}$ \\ \cline{1-7}
$\lambda=\frac{1}{2}$ & $0.2537$ & No & $9.0867\cdot 10^{-4}$ & - & $66$ &  $1^{M}20^{S}$ \\ \cline{1-7}
 $\lambda=\frac{3}{4}$  & $0.3543$  & Approximately&  $0.0015$ & $(\pm 0.9960,\pm 0.9960)$ & $74$ & $1^{M}5^{S}$ \\ \cline{1-7}
 $\lambda=1$  & $0.2870$ &  No & 0.0013 & - &  $115$ & $2^{M}$ \\ \cline{1-7}
\hline 
\end{tabular}
\caption{Data of the Algorithm \ref{upper} applyed to the Motzkin polynomial}
\label{table:1}
\end{table}

\end{ej}

\begin{ej} Let us consider the problem of minimizing the following polynomial taken from \cite[Example 3]{las}:
\begin{equation*}
\begin{aligned}
& {\text{minimize}}
& & f(x)=x_{1}^{2}x_{2}^{2}(x_{1}^{2}+x_{2}^{2}-1)\\
& \text{subject to}
& & x\in\R^{2}\\
\end{aligned}
\end{equation*}
Since it holds $-1\leq X_{1}^2+X_{2}^{2}-1$ it also holds that $-1\leq X_{1}^{2}X_{2}^{2}(X_{1}^{2}+X_{2}^{2}-1)$, or what is the same $f+1$ is nonnegative. $f+1$ is also not a sum of squares as one can easily check with the so called \textit{Gram-matrix method}, see \cite[Lemma 3.8]{lau}. In fact one can compute with Calculus that $P^{*}=f(\pm \frac{1}{\sqrt{3}},\pm\frac{1}{\sqrt{3}})=\frac{1}{27}$. Let us see how the algorithm works with this polynomial $f$. An optimal solution $z\in\R^{s_{6}}$ of the Moment relaxation $(P_{6})$, has the following Moment matrix:
aqui
{\scriptsize
\begin{align} \nonumber
M_{3}(z)= \begin{blockarray}{cccccccc}
\text{ } & 1 & X_{1} & X_{2} & X_{1}^2 & X_{1}X_{2} & X_{2}^2 \\
\begin{block}{c(ccccccc}
1&1.00&0.00&0.00&2270.73&-0.00&2270.73\\
X_{1}&0.00&2270.73&-0.00&0.00&0.00&0.00&\\
X_{2}&0.00&-0.00&2270.73&0.00&0.00&0.00\\
X_{1}^{2}&2270.73&0.00&0.00&10462942.40&0.00&411.34\\
X_{1}X_{2}&-0.00&0.00&0.00&0.00&411.34&-0.00\\
X_{2}^{2}&2270.73&0.00&0.00&411.34&-0.00&10462943.55\\
X_{1}^{3}&0.00&10462942.40&0.00&2.13&0.00&0.00\\
X_{1}^{2}X_{2}&0.00&0.00&411.34&0.00&0.00&0.00\\
X_{1}X_{2}^{2}&0.00&411.34&-0.00&0.00&0.00&0.00\\
X_{2}^{3}&0.00&-0.00&10462943.55&0.00&0.00&2.50 \\ 
\end{block}
   \end{blockarray}\\ \nonumber
\begin{blockarray}{ccccccc}
X_{1}^{3} & X_{1}^{2}X_{2} & X_{1}X_{2}^{2} & X_{2}^{3} & \text{ }\\
\begin{block}{cccc)ccc}
0.00&0.00&0.00&0.00&1\\
10462942.40&0.00&411.34&-0.00&X_{1}\\
0.00&411.34&-0.00&10462943.55&X_{2}\\
2.13&0.00&0.00&0.00&X_{1}^{2}\\
0.00&0.00&0.00&0.00&X_{1}X_{2}\\
0.00&0.00&0.00&2.50&X_{2}^{2}\\
165266760920.34&0.00&116.88&0.00&X_{1}^{3}\\
0.00&116.88&0.00&116.88&X_{1}^{2}X_{2}\\
116.88&0.00&116.88&-0.00&X_{1}X_{2}^{2}\\
0.00&116.88&-0.00&165266795981.98&X_{2}^{3}\\
\end{block}
   \end{blockarray}\\ \nonumber
\end{align}}

This matrix is neither flat nor $\widetilde{M_{3}(z)}$ is a generalized Hankel matrix so we can not conclude optimality \ref{above}, in fact the value of this optimal solution is $P^{*}_{6}=-177.5859<<P^{*}=-\frac{1}{27}\approx -0.0370$ stricly smaller than the minimum. According to the Algorithm \ref{upper} we compute the matrices $B_{M}$, $T_{M}$ associated to a feasible solution $M$ and finally we solve the semidefinite program $(P_{M,\frac{1}{100}})$ and we get the following matrix:

{\scriptsize\begin{align}A_{M}:=M_{3}(y_{0})\nonumber= \begin{blockarray}{ccccccccccc}
\text{ } & 1 & X_{1} & X_{2} & X_{1}^2 & X_{1}X_{2} & X_{2}^2 & X_{1}^{3} & X_{1}^{2}X_{2} & X_{1}X_{2}^{3} & X_{2}^{3} \\
\begin{block}{r(rrrrrrrrrr)}
1&1.00&-0.00&-0.00&0.41&-0.00&0.41&0.00&0.00&0.00&0.00\\
X_{1}&-0.00&0.41&-0.00&0.00&0.00&0.00&0.17&-0.00&0.16&-0.00\\
X_{2}&-0.00&-0.00&0.41&0.00&0.00&0.00&-0.00&0.16&-0.00&0.17\\
X_{1}^{2}&0.41&0.00&0.00&0.17&-0.00&0.16&-0.00&-0.00&0.00&0.00\\
X_{1}X_{2}&-0.00&0.00&0.00&-0.00&0.16&-0.00&-0.00&0.00&0.00&-0.00\\
X_{2}^{2}&0.41&0.00&0.00&0.16&-0.00&0.17&0.00&0.00&-0.00&-0.00\\
X_{1}^{3}&0.00&0.17&-0.00&-0.00&-0.00&0.00&0.07&-0.00&0.07&-0.00\\
X_{1}^{2}X_{2}&0.00&-0.00&0.16&-0.00&0.00&0.00&-0.00&0.07&-0.00&0.07\\
X_{1}X_{2}^{2}&0.00&0.16&-0.00&0.00&0.00&-0.00&0.07&-0.00&0.07&-0.00\\
X_{2}^{3}&0.00&-0.00&0.17&0.00&-0.00&-0.00&-0.00&0.07&-0.00&0.07\\
\end{block}
   \end{blockarray}\\ \nonumber
	\end{align}}
We diagonalize this matrix to compute the rank and we get that the  eigenvalues are:
\begin{equation}
\{1.3284, 0.5383, 0.58383,  0.1630, 0.0024, 0.0008, 0.0008, 0.0000, 0.0000, 0.0000 \}\nonumber
\end{equation}
what implies, after rounding to the fourth decimal, that the $\rk A_{M}=6$ and therefore $A_{M}$ is flat and since we got that $U=\sum_{|\alpha|\leq 6}f_{\alpha}y_{0,\alpha}=-0.0305$ then we can deduce:
\begin{equation}
P^{*}\in [-177.5859,-0.0305 ]\nonumber
\end{equation}
 So we proceed to extract minimizers with the Algortihm defined in \cite[Algorithm 1]{mlq1} and we get that:
\begin{align}
A_{M}= & \nonumber  0.2501\cdot V_{3}V_{3}^{T}(-0.6354,-0.6354)+ 0.2499\cdot V_{3}V_{3}^{T}(-0.6354,0.6354)\\ \nonumber
&+0.2499\cdot V_{3}V_{3}^{T}(-0.6354,0.6354)+ 0.2501\cdot V_{3}V_{3}^{T}(0.6354,0.6354) \nonumber
\end{align}
In conclusion we get that:
\begin{equation}
 f(\pm{0.6354},\pm{0.6354})=-0.0305\geq P^{*}=-0.0370=f(\pm{0.5774},\pm{0.5774})\nonumber
\end{equation}
The table \ref{table:2G} shows the different data that we get by solving the semidefinite program $(P)_{M,\lambda}$ for different values of $\lambda\in [0,1]$. 

\begin{table}
\begin{tabular}{ |c|c|c|c|c|c|c| } 
\hline
 & $U$ & Flat & $||A_{M}-B_{M}||$ & Potential Minimizers & Iterations & Time \\ \cline{1-7}
 $\lambda=\frac{1}{1000}$ & $-4.4169$ & No & $7.8542\cdot10^{4}$ & - & $76$ &  $2^{M}30^{S}$ \\ \cline{1-7}
 $\lambda=\frac{1}{100}$ & $-0.0305$  & Aproximately & $7.928\cdot 10^{-4}$ & $(\pm 0.6354,\pm 0.6354)$ & $85$ &  $3^{M}3^{S}$ \\ \cline{1-7}
$\lambda=\frac{1}{60}$ & $-0.0255$ & Yes & $9.9708\cdot 10^{-5}$ & $(\pm 0.6566 ,\pm 0.6566 )$ & $76$ &  $2^{M}50^{S}$ \\ \cline{1-7}
 $\lambda=\frac{1}{4}$ & $0.0802$ & No & $0.0627$ & - & $50$ &  $1^{M}10^{S}$ \\ \cline{1-7}
 $\lambda=\frac{1}{2}$  & $0.1634$  & Yes &  $9.5965\cdot 10^{-5}$ & $(\pm 0.8233,\pm 0.8233)$ & $123$ & $3^{M}$ \\ \cline{1-7}
 $\lambda=\frac{3}{4}$ & $0.2399$ & No & $0.0430$ & - & $50$ &  $1^{M}20^{S}$ \\ \cline{1-7}
 $\lambda=1$  & $0.1331$ &  No & 0.0103 & - &  $84$ & $3^{M}11^{S}$ \\ \cline{1-7}
\hline 
\end{tabular}
\caption{Data of the Algorithm \ref{upper} for the polynomial $x_{1}^{2}x_{2}^{2}(x_{1}^{2}+x_{2}^{2}-1)$}
\label{table:2G}
\end{table}
\end{ej}

\section{Algorithm 2 based on the Nie, Demmel and Sturmfels method}

In the Algorithm \ref{upper} the starting matrix $M:=M_{d}(y)$ was a Generalized Hankel matrix positive semidefinite associated to a feasible solution $y\in\R^{s_{d}}$ of the Moment relaxation $(P_{\text{deg}f})$. If instead we start with a matrix associated to an optimal solution of the NDS relaxation, see the definition in \ref{NDS} taken from \cite{NDS}. Then by minimizing $E$ in \eqref{amy} we attempt to simultaneously minimize the distance of the new optimal solution $A$ to $M$ and at the same time we try to control the rank of the matrix $A$. In this section we see some experimental examples using this technique. 

For $p\in\R[\underline{X}]_{k}$ denote $d_{p}:=k-\deg{p}$  and consider the following vector: 
\begin{equation}\label{localizing}
\begin{aligned}
    pV_{d_{p}}:= (&p1,pX_{1}, pX_{2}, \ldots,  pX_{n},pX_{1}^2,pX_{1}X_{2},\ldots,pX_{1}X_{n},pX_{2}^2,\\
		&pX_{2}X_{3},\ldots,pX_{2}^2,\ldots,pX_{1}^{d_{p}},\ldots,pX_{n-1}X_{n}^{d_{p}-1},pX_{n}^{d_{p}})^{T} 
\end{aligned}
\end{equation}

\begin{defi}
For  $p\in\R[\underline{X}]_{k}$ the localizing vector of $p$ of degree $k$ is the vector  resulting from substituting every monomial $\underline{X}^{\alpha}$ such that $|\alpha|\leq k$ in \eqref{localizing} for a new variable $Y_{\alpha}$. We denote this vector by $V_{k,p}\in\R[\underline{Y}]_{1}^{s_{d_{p}}}$.
\end{defi}

Providing that $P^{*}=f(x)$ for some $x\in\R^{n}$, by Calculus we know that the local and global minima are contained in the real gradient variety:
\begin{equation}
\mathcal{V}_{f}:=\{u\in\R^{n}\text{ }|\text{ } \nabla f(u)=0 \}\nonumber
\end{equation}

In \cite{NDS}, it is considered to add to the problem \eqref{explanationauxiliar} the constraints:
\begin{equation}
p\frac{\partial f}{\partial X_i}(x)=0\text{ for all } p\in\R[\underline{X}]_{2d-\text{deg}f+1}.\nonumber
\end{equation}
namely:
\begin{align}
 \text{   minimize } \nonumber  f(x)  &\text{ subject to }\\   \nonumber
& x \in\R^{n},\\ \nonumber
& p(x)^{2}\geq 0 \text{ for all }p\in\R[\underline{X}]_{d}\text{ and}  \\ \nonumber
& p\frac{\partial f}{\partial X_i}(x)=0\text{ for all } p\in\R[\underline{X}]_{2d-\text{deg}f+1}. \nonumber
\end{align}

After linearization, or in other words, after substitute every monomial  $\underline{X}^{\alpha}$  for a new variable $Y_{\alpha}$,  this turned out into the following hierarchy of semidefinite programs.

\begin{dyn}
Let $(P)$ be a polynomial optimization problem  as in \eqref{Popw} such that there exists $x\in\R^{n}$ with $P^{*}=f(x)$ and 
let $k\in$ $\N_{0}\cup\{\infty\}$ such that $f\in\R[\underline{X}]_{k}$. We call the $\text{NDS}$ relaxation  of $(P)$ of degree $k$ to the following semidefinite optimization problem:
\begin{align}\label{NDS}
(P_{k,\text{NDS}}) \text{ minimize } \sum_{|\alpha|\leq k} f_{\alpha}y_{\alpha}\nonumber & \text{ subject to: } \\
& y_{\left(0,\ldots,0\right)}=1,\\  \nonumber
& M_{\lfloor\frac{\deg{k}}{2}\rfloor}(y)\succeq 0 \text{ and}\\  \nonumber
& V_{k,\frac{\partial f}{\partial X_{i}}}(y)=0\text{ for all } i\in\{1,\ldots,n\} \nonumber
\end{align}
the optimal value of $(P_{k,\text{NDS}})$ that is to say, the infimum over all
\begin{equation}
 y=(y_{\left(0,\ldots,0 \right)},\ldots,y_{\left(0,\ldots,k \right)})\in\R^{s_{k}}\nonumber
\end{equation}
that ranges over all feasible solutions of $(P_{k,\text{NDS}})$ is denoted by $P^{*}_{k,\text{NDS}}\in\{-\infty\}\cup\R\cup\{\infty\} $.
\end{dyn}

Let us remember a few properties about the convergence of this hierarchy of relaxations. For all the details about the convergence of the $\text{NDS}$ relaxation hierarchy we refer to the reader to \cite{NDS}. 

\begin{prop}\label{calor}
Let $(P)$ be a polynomial optimization problem as in $\eqref{Popw}$ and let $k\in\N_{0}\cup\{ \infty \}$ such that $f\in\R[\underline{X}]_{k}$. Set  $d:=\lfloor\frac{\deg{k}}{2}\rfloor$. The following holds:
\begin{enumerate}
\item $P^{*}\geq\ldots\geq P^{*}_{k+1,\text{NDS}}\geq P^{*}_{k,\text{NDS}}$
\item  Every matrix of the form:
\begin{equation}
 M=\sum_{i=1}^{N}\lambda_{i}V_{d}V_{d}^{T}(a_{i})\in S\R^{s_{d} \times s_{d}}\nonumber
 \end{equation}
  with $a_{i}\in \mathcal{V}_{f}$ for all $i\in\{1,\ldots,N\}$ and with $\sum_{i=1}^{N}\lambda_{i}=1$ is the Moment matrix of a feasible solution $y\in\R^{s_{k}}$ of $(P_{k,\text{NDS}})$ and $\sum_{|\alpha|\leq k} f_{\alpha}y_{\alpha}\geq P^{*}$.
  \item If $(P_{k,\text{NDS}})$ has an optimal solution $y\in\R^{s_{k}}$ such that there exists $N\in\N$ and $a_{1},\ldots,a_{N}\in \mathcal{V}_{f}$ and $\lambda_{1}>0,\ldots,\lambda_{N}>0$ such that:
	\begin{equation}
	M_{d}(y)=\sum_{i=1}^{N}\lambda_{i}V_{d}V_{d}^{T}(a_{i})\in S\R^{s_{d} \times s_{d}}\nonumber
	\end{equation}
	Then $P^{*}=P^{*}_{k}$ and $a_{1},\ldots,a_{r}$ are minimizers of $f$.
\end{enumerate}
\end{prop}
\begin{proof}
For a proof of this Proposition we refer the reader to \cite[Proposition 3.9]{mlq1}.
\end{proof}

The Theorem \ref{above} also holds for the polynomial optimization problems with constraints adding an extra condition that the nodes should belong to the semialgebraic set, in this case to the real gradient variety.

\begin{teor}\label{above2}
Let $f\in\R[\underline{X}]_{k}$ and $(P)$ be the polynomial optimization problem without constraints defined in \eqref{Popw}, and let $y\in\R^{s_{k}}$ be an optimal solution of the Moment relaxation $(P_{k,\text{NDS}})$ and set $d:=\left\lfloor{\frac{k}{2}}\right\rfloor$. Then the following conditions hold:
\begin{enumerate}
\item If $\widetilde{M_{d}(y)}$ is Generalized Hankel and $f\in\R[\underline{X}]_{k-1}$  then there exist $\lambda_{1}>0,\ldots,\lambda_{r}>0$ and $a_{1},\ldots,a_{N}\in\R^{n}$ such that:
\begin{equation}
\widetilde{M_{d}(y)}=\sum_{i=1}^{N}\lambda_{i}V_{d}V_{d}^{T}(a_{i})\in S\R^{s_{d} \times s_{d}}\nonumber
\end{equation}
In case $a_{1},\ldots,a_{N}\in \mathcal{V}_{f}$ then $P^{*}=P^{*}_{k,\text{NDS}}$, and $a_{1},\ldots,a_{N}$ are minimimizers of $(P)$.
\item If $M_{d}(y)$ is flat, then there exists  $\lambda_{1}>0,\ldots,\lambda_{N}>0$ and $a_{1},\ldots,a_{N}\in\R^{n}$ such that:
\begin{equation}
M_{d}(y)=\sum_{i=1}^{N}\lambda_{i}V_{d}V_{d}^{T}(a_{i})\in S\R^{s_{d} \times s_{d}}\nonumber
\end{equation}
If $a_{1},\ldots,a_{N}\in \mathcal{V}_{f}$ then $P^{*}=P^{*}_{k,\text{NDS}}$ and $a_{1},\ldots,a_{N}$ are minimimizers of $(P)$.
\end{enumerate}
\end{teor} 

\begin{proof}
The proof is in  \cite[Theorem 7.1]{mlq1}
\end{proof}

\newpage

\begin{algorithm}\label{upper2}
  \KwIn{A polynomial optimization problem $(\textbf{P})$ \eqref{Pop} without constraints and a strategic $\lambda\in [0,1]$.}
  \KwOut{An upper bound $U\geq\textbf{P}^{*}$ and if possible, points $\textbf{a}_{1},\ldots,\textbf{a}_{r}\in\mathcal{V}_{f}$ such that $f(\textbf{a}_{\textbf{i}})=U$.}\
Compute an optimal solution of the $\text{NDS}$ relaxation $(P_{\deg f,\text{NDS}})$. Denote it by $y\in\R^{s_{\deg f}}$ and set $k:=\deg f$ and $d:=\lfloor\frac{\deg{f}}{2}\rfloor$ and $\textbf{M}:=M_{d}(y)$, . \;
Take a maximum linearly independent subset of the first $s_{d-1}$ columns of $\textbf{M}$ and denote it by $\{M_{1},\ldots,M_{l}\}$ and set $C:=(M_{1}\cdots M_{l})$. \;
Define the folowing matrix using the definition of the solution of the least squares problem \eqref{vectorprojection},
$\textbf{B}_{\textbf{M}}:=(M_{1}\cdots M_{s_{d-1}}|Cx^{*}_{C,M_{s_{d-1}+1}}\cdots Cx^{*}_{C,M_{s_{d}}})$.\;
Finally define the following matrix:
\begin{equation}
\textbf{T}_{\textbf{M}}:=\left(
\begin{array}{c|c}
I_{s_{d}s_{d}}& \begin{array}{c} (A_{M})_{1,1}-(B_{M})_{1,1}\\ \vdots\\  (A_{M})_{s_{d},s_{d}}-(B_{M})_{s_{d},s_{d}} \end{array} \\ \cline{1-2}

\begin{array}{ccc}(A_{M})_{1,1}-(B_{M})_{1,1} & \cdots  & (A_{M})_{s_{d},s_{d}}-(B_{M})_{s_{d},s_{d}} \end{array} & E||M-B_{M}||^{2}
   \end{array}
    \right)\nonumber
\end{equation}
where the unknows are $\textbf{E}$ and the entries of the matrix $\textbf{A}_{\textbf{M}}\in S\R^{s_{d}\times s_{d}}$. \;
Solve the following positive semidefinite program:
\begin{align}
(P_{M,\lambda,\text{NDS}}) \text{   minimize } \nonumber E\lambda+(1-\lambda)\sum_{|\alpha|\leq k}f_{\alpha}y_{0,\alpha}  & \text{ subject to }\\  \nonumber
& \textbf{A}_{\textbf{M}}=M_{d}(y_{0})\\ \nonumber
& E \in\R \text{ , } T_{M}\succeq 0\\  \nonumber
& \textbf{A}_{\textbf{M}}\succeq 0 \\ \nonumber
& V_{k,\frac{\partial f}{\partial X_{i}}}(y_{0})=0\text{ for all }i\in\{1,\ldots,n\}\nonumber
\end{align}
Note that here the decision variables are $E$ and the entries of the matrix $\textbf{A}_{\textbf{M}}\in S\R^{s_{d}\times s_{d}}$ which are $y_{0}\in\R^{s_{k}}$. \;
\If {$\textbf{A}_{\textbf{M}}$ is flat or $||\textbf{A}_{\textbf{M}}-\textbf{B}_{\textbf{M}}||$ small enough} { $\textbf{U}:=\sum_{|\alpha|\leq k}f_{\alpha}y_{0,\alpha}$ and extract minimizers if possible with \cite[Algorithm 1]{mlq1}.}
\Else{$\textbf{M}:=\textbf{A}_{\textbf{M}}$ and go to $\textbf{3}$} 

  \caption{Finding an upper bound $U$ for $(P)$ \eqref{Popw}, $U\geq P^{*}$ using the $NDS$ relaxation and if possible potential minimizers or minimizers}
  
\end{algorithm}

\begin{rem}
In Algorithm \ref{upper2} since $y_{0}\in\R^{s_{k}}$ is a feasible solution of $(P_{k,\text{NDS}})$ then it holds $U\geq P^{*}_{k,\text{NDS}}$. Moreover if $\textbf{A}_{\textbf{M}}$ is flat then by \ref{above2} $(2)$, there exits $a_{1},\ldots,a_{N} \in \R^{n}$ and $\lambda_{1}>0,\ldots,\lambda_{N}>0$ such that:
\begin{equation}
A_{M}=\sum_{i=1}^{N}\lambda_{i}V_{d}V_{d}^{T}(a_{i})\in S\R^{s_{d} \times s_{d}}\nonumber
\end{equation}
But in contrast with Algorithm \ref{upper} we need to check that $a_{1},\ldots,a_{N}\in\mathcal{V}_{f}$ to apply \ref{calor} $(2)$ in order to conclude:
\begin{equation}
 P^{*}\in [P^{*}_{k,\text{NDS}},U]\nonumber
\end{equation}
\end{rem}

\section{Numerical results of Algorithm 2}

\begin{ej}
Let us consider the  problem of minimizing the so called \textit{Robinson polynomial}:
\begin{equation}
X_{1}^{6}+X_{2}^{6}+1-(X_{1}^{4}X_{2}^{2}+X_{2}^{4}+X_{1}^{2}+X_{1}^{2}X_{2}^{4}+X_{2}^{2}+X_{1}^{2})+3X_{1}^{2}X_{2}^{2}\nonumber
\end{equation}
 a nonnegative polynomial which is not sum of squares (see \cite{reznick,Robinson}):
\begin{equation}
\begin{aligned}
& {\text{minimize}}
& & f(x)=x_{1}^{6}+x_{2}^{6}+1-(x_{1}^{4}x_{2}^{2}+x_{2}^{4}+x_{1}^{4}+x_{1}^{2}x_{2}^{4}+x_{2}^{2}+x_{1}^{2})+3x_{1}^{2}x_{2}^{2}\\  \nonumber
& \text{subject to}
& & x\in\R^{2}\\
\end{aligned}
\end{equation}
This polynomial attains the minimum in the $\text{NDS}$ relaxation of degree 8, since we get the following optimal solution $z\in\R^{s_{8}}$\eqref{madremia}:\\

{\scriptsize
\begin{align}\label{madremia}
 M_{4}(z)=\begin{blockarray}{ccccccccccccccccc}
\text{ } & 1 & X_{1} & X_{2} & X_{1}^2 & X_{1}X_{2} & X_{2}^2 & X_{1}^{3} & X_{1}^{2}X_{2} & X_{1}X_{2}^{3} & X_{2}^{3}  \\
\begin{block}{r(rrrrrrrrrrrrrrrr} 
1&1.00&-0.00&0.00&0.64&0.00&0.64&-0.00&0.00&0.00&0.00\\
X_{1}&-0.00&0.64&0.00&-0.00&0.00&0.00&0.64&0.00&0.29&0.00\\
X_{2}&0.00&0.00&0.64&0.00&0.00&0.00&0.00&0.29&0.00&0.64\\
X_{1}^{2}&0.64&-0.00&0.00&0.64&0.00&0.29&-0.00&0.00&0.00&0.00\\
X_{1}X_{2}&0.00&0.00&0.00&0.00&0.29&0.00&0.00&0.00&0.00&0.00\\
X_{2}^{2}&0.64&0.00&0.00&0.29&0.00&0.64&0.00&0.00&0.00&0.00\\
X_{1}^{3}&-0.00&0.64&0.00&-0.00&0.00&0.00&0.64&0.00&0.29&0.00\\
X_{1}^{2}X_{2}&0.00&0.00&0.29&0.00&0.00&0.00&0.00&0.29&0.00&0.29\\
X_{1}X_{2}^{2}&0.00&0.29&0.00&0.00&0.00&0.00&0.29&0.00&0.29&0.00\\
X_{2}^{3}&0.00&0.00&0.64&0.00&0.00&0.00&0.00&0.29&0.00&0.64\\
X_{1}^{4}&0.64&-0.00&0.00&0.64&0.00&0.29&-0.00&0.00&0.00&0.00\\
X_{1}^{3}X_{2}&0.00&0.00&0.00&0.00&0.29&0.00&0.00&0.00&0.00&0.00\\
X_{1}^{2}X_{2}^{2}&0.29&0.00&0.00&0.29&0.00&0.29&0.00&0.00&0.00&0.00\\
X_{1}X_{2}^{3}&0.00&0.00&0.00&0.00&0.29&0.00&0.00&0.00&0.00&0.00\\
X_{1}^{4}&0.64&0.00&0.00&0.29&0.00&0.64&0.00&0.00&0.00&0.00\\
\end{block}
   \end{blockarray}\\ \nonumber
\begin{blockarray}{ccccccccccccccccc}
X_{1}^{4} & X_{1}^{3}X_{2} & X_{1}^{2}X_{2}^{2} & X_{1}X_{2}^{3} & X_{2}^{4} & \text{ }\\
\begin{block}{ccccc)cccccccccccc}
0.64&0.00&0.29&0.00&0.64&1\\
-0.00&0.00&0.00&0.00&0.00&X_{1}\\
0.00&0.00&0.00&0.00&0.00&X_{2}\\
0.64&0.00&0.29&0.00&0.29&X_{1}^{2}\\
0.00&0.29&0.00&0.29&0.00&X_{1}X_{2}\\
0.29&0.00&0.29&0.00&0.64&X_{2}^{2}\\
-0.00&0.00&0.00&0.00&0.00&X_{1}^{3}\\
0.00&0.00&0.00&0.00&0.00&X_{1}^{2}X_{2}\\
0.00&0.00&0.00&0.00&0.00&X_{1}X_{2}^{2}\\
0.00&0.00&0.00&0.00&0.00&X_{2}^{3}\\
1.01&-0.00&0.65&-0.00&0.65&X_{1}^{4}\\
-0.00&0.65&-0.00&0.65&-0.00&X_{1}^{3}X_{2}\\
0.65&-0.00&0.65&-0.00&0.65&X_{1}^{2}X_{2}^{2}\\
-0.00&0.65&-0.00&0.65&-0.00&X_{1}X_{2}^{3}\\
0.65&-0.00&0.65&-0.00&1.01&X_{2}^{4}\\
\end{block}
   \end{blockarray}\\ \nonumber
\end{align}}

with optimal value $P^{*}_{8,\text{NDS}}=1.3558\cdot 10^{-10}\approx 0$. We can conclude optimality since $\widetilde{M_{4}(y)}$ is a generalized Hankel matrix, or equivalently, the truncated multiplication operators commute and we can get with \cite[Algorithm 1]{mlq1} 8 potential minimizers:
\begin{equation}
(\pm 1 ,\pm 1),(\pm 1, 0),(0,\pm 1)\in \mathcal{V}_{f}.\nonumber
\end{equation}
 Therefore with the $\text{NDS}$ relaxation of degree 8 by Theorem \ref{calor} $(1)$ we have: 
\begin{equation}
P^{*}=P^{*}_{8,\text{NDS}}=f(\pm 1,\pm 1)=f(\pm 1,0)=f(0,\pm 1)=0.\nonumber
\end{equation}
Let us now use the Algorithm \ref{upper2} to see if we can get an upper bound for the minimum and potential miminimizers starting with a moment matrix of lower dimension which is  an optimal solution of the $\text{NDS}$ relaxation of degree $6$. Namely, we start with the moment matrix:

\begin{align}
\textbf{M}:=M_{3}(y)=\begin{tiny}\begin{blockarray}{ccccccccccc}\nonumber
\text{ } & 1 & X_{1} & X_{2} & X_{1}^2 & X_{1}X_{2} & X_{2}^2 & X_{1}^{3} & X_{1}^{2}X_{2} & X_{1}X_{2}^{3} & X_{2}^{3} \\
\begin{block}{r(rrrrrrrrrr)}
1&1.00&0.00&0.00&1.39&0.00&1.39&0.00&0.00&-0.00&0.00\\
X_{1}&0.00&1.39&0.00&0.00&0.00&-0.00&2.37&0.00&1.50&0.00\\
X_{2}&0.00&0.00&1.39&0.00&-0.00&0.00&0.00&1.50&0.00&2.37\\
X_{1}^{2}&1.39&0.00&0.00&2.37&0.00&1.50&-0.00&0.00&-0.00&0.00\\
X_{1}X_{2}&0.00&0.00&-0.00&0.00&1.50&0.00&0.00&-0.00&0.00&-0.00\\
X_{2}^{2}&1.39&-0.00&0.00&1.50&0.00&2.37&-0.00&0.00&-0.00&0.00\\
X_{1}^{3}&0.00&2.37&0.00&-0.00&0.00&-0.00&3946.00&-0.00&3945.46&-0.00\\
X_{1}^{2}X_{2}&0.00&0.00&1.50&0.00&-0.00&0.00&-0.00&3945.46&-0.00&3945.46\\
X_{1}X_{2}^{2}&-0.00&1.50&0.00&-0.00&0.00&-0.00&3945.46&-0.00&3945.46&-0.00\\
X_{2}^{3}&0.00&0.00&2.37&0.00&-0.00&0.00&-0.00&3945.46&-0.00&3946.00\\
\end{block}
\end{blockarray}
\end{tiny}
\end{align}

with optimal value $P^{*}_{6,\text{NDS}}=-0.9333$. After applying the Algorithm \eqref{upper2} we get the following results, see Table \ref{table:3G}. Here we obviously get that $P^{*}_{6,\text{NDS}}=-0.9333\leq P_{M,\lambda,\text{NDS}}^{*}$, but we can no certify the belonging to any interval since unfortunately we do not get neither flat nor the truncated multiplication operators commute. 

\begin{table}
\begin{tabular}{ |c|c|c|c|c|c|c| } 
\hline
 & $U$ & Flat & $||A_{M}-B_{M}||$ & Potential Minimizers & Iterations & Time \\ \cline{1-7}
 $\lambda=\frac{1}{100}$ & $-0.9278$  & No & $7.329\cdot 10^{3}$ & - & $10$ &  $33^{S}$ \\ \cline{1-7}
 $\lambda=\frac{1}{60}$ & $-0.9288$ & No & $7.1978\cdot 10^{3}$ & - & $10$ &  $40^{S}$ \\ \cline{1-7}
 $\lambda=\frac{1}{30}$  & $-0.5709$  & No &  $6.8149\cdot 10^{3}$ & - & $10$ & $50^{S}$ \\ \cline{1-7}
 $\lambda=\frac{1}{10}$ &-0.0159 & No & $6.1421\cdot 10^{3}$ & - & $10$ & $48^{S}$\\ \cline{1-7}
 $\lambda= \frac{1}{5}$ & $0.0549$ & No &$5.9942\cdot 10^{3} $& -& $10$ & $36^{S}$\\ \cline{1-7} 
$\lambda= \frac{1}{4}$ & $0.0670$ & No &$6.0629\ddot 10^{3}$& -& $500$ & $40^{M}$\\ \cline{1-7}
 $\lambda=\frac{3}{4}$  & $0.1018$ &  No & $5.8919\cdot 10^{3}$ & - &  $10$ & $38^{S}$ \\ \cline{1-7}
 $\lambda = 1$ & $1.060$& No &$5.7389\cdot 10^{3}$ &-&$10$& $36^{S}$\\ \cline{1-7}
\hline 
\end{tabular}
\caption{Data of the Algorithm \ref{upper2} for the Robinson polynomial}
\label{table:3G}
\end{table}
\end{ej}

\begin{ej}
Let us consider again the problem of minimizing the Motzkin polynomial:
\begin{equation}
\begin{aligned}
& {\text{minimize}}
& & f(x)=x_{1}^{4}x_{2}^{2}+x_{1}^2x_{2}^{4}-3x_{1}^2x_{2}^2+1 \\  \nonumber
& \text{subject to}
& & x\in\R^{2}\\
\end{aligned}
\end{equation}

As it was mentioned in \cite{NDS} this polynomial attains the minimum in the relaxation $(P_{8,\text{NDS}})$. After applying the Algorithm \ref{upper2} we get the following information, see Table \ref{table:4G}. Notice that we start with an optimal solution, associated to the matrix $M$, of the relaxation $(P_{6,\text{NDS}})$. More precise we start the algorithm with the following matrix:
{\scriptsize
\begin{align}
M:=M_{3}(y)=\begin{blockarray}{cccccccc}\nonumber
\text{ } & 1 & X_{1} & X_{2} & X_{1}^2 & X_{1}X_{2} & X_{2}^2 \\
\begin{block}{c(ccccccc}
1&1.00&-0.00&-0.00&663.07&-0.00&663.07\\
X_{1}&-0.00&663.07&-0.00&-0.00&-0.00&0.00\\
X_{2}&-0.00&-0.00&663.07&-0.00&0.00&-0.00\\
X_{1}^{2}&663.07&-0.00&-0.00&891655.96&-0.00&423.13\\
X_{1}X_{2}&-0.00&-0.00&0.00&-0.00&423.13&-0.00\\
X_{2}^{2}&663.07&0.00&-0.00&423.13&-0.00&891655.96\\
X_{1}^{3}&-0.00&891655.96&-0.00&-0.00&0.00&-0.00\\
X_{1}^{2}X_{2}&-0.00&-0.00&423.13&0.00&-0.00&-0.00\\
X_{1}X_{2}^{2}&0.00&423.13&-0.00&-0.00&-0.00&0.00\\
X_{2}^{3}&-0.00&-0.00&891655.96&-0.00&0.00&-0.00\\
\end{block}
   \end{blockarray}\\ \nonumber
	\begin{blockarray}{ccccccc}
X_{1}^{3} & X_{1}^{2}X_{2} & X_{1}X_{2}^{2} & X_{2}^{3} & \text{ }\\
\begin{block}{cccc)ccc}
-0.00&-0.00&0.00&-0.00&1\\
891655.96&-0.00&423.13&-0.00&X_{1}\\
-0.00&423.13&-0.00&891655.96&X_{2}\\
-0.00&0.00&-0.00&-0.00&X_{1}^{2}\\
0.00&-0.00&-0.00&0.00&X_{1}X_{2}\\
-0.00&-0.00&0.00&-0.00&X_{2}^{2}\\
4116203733.28&-0.00&423.13&-0.00&X_{1}^{3}\\
-0.00&423.13&-0.00&423.13&X_{1}^{2}X_{2}\\
423.13&-0.00&423.13&-0.00&X_{1}X_{2}^{2}\\
-0.00&423.13&-0.00&4116203751.60&X_{2}^{3}\\
\end{block}
   \end{blockarray}\\ \nonumber
\end{align}}

with associated optimal value $P^{*}_{6,\text{NDS}}=-422,13$.
\begin{table}
\begin{tabular}{ |c|c|c|c|c|c|c| } 
\hline
 & $U$ & Flat & $||A_{M}-B_{M}||$ & Potential Minimizers & Iterations & Time \\ \cline{1-7}

 $\lambda=\frac{1}{1000}$ & -30.6672  & No & $1.1486\cdot 10^{6}$ & - & $50$ &  $4^{M}42^{S}$ \\ \cline{1-7}
 $\lambda=\frac{1}{100}$ & 0.1688  & No & $1.1486\cdot 10^{6}$ & - & $10$ &  $50^{S}$ \\ \cline{1-7}
 $\lambda=\frac{1}{60}$ & $0.2458$ & No & $1.4886\cdot 10^{6}$ & - & $10$ &  $40^{S}$ \\ \cline{1-7}
$\lambda=\frac{1}{4}$ & $-0.8789$ & No & $1.4887\cdot 10^{6}$ & - & $50$ &  $4^{M}$ \\ \cline{1-7}
 $\lambda=\frac{1}{2}$  & $0.4232$  & No &  $1.1485\cdot 10^{6}$ & - & $10$ & $41^{S}$ \\ \cline{1-7}
$\lambda=\frac{3}{4}$ & 1 & Yes & $1.1629\cdot 10^{6}$ & $(0,\pm 44.9350),(\pm 44.9350,0)$ & $50$ & $4^{M}42^{S}$\\ \cline{1-7}
 $\lambda=1$ & 1 & Yes & $1.1629\cdot 10^{6}$ & $(0,\pm 44.9352),(\pm 44.9352,0)$ & $20$ & $1^{M}20^{S}$\\ \cline{1-7}
\hline 
\end{tabular}
\caption{Data of the Algorithm \ref{upper2} for the Motzkin polynomial}
\label{table:4G}
\end{table}

As we can see in the Table \ref{table:4G}, we get $P^{*}_{M,1,\text{NDS}}=1$ after $20$ iterations of the Algorithm \ref{upper2}. Namely we get the following flat optimal solution of $(P_{M,1,\text{NDS}})$:

{\scriptsize
\begin{align}
A_{M}=M_{3}(y_{0})=
\begin{blockarray}{cccccccc}\nonumber
\text{ } & 1 & X_{1} & X_{2} & X_{1}^2 & X_{1}X_{2} & X_{2}^2 \\
\begin{block}{c(ccccccc}
1&1.00&0.00&-0.00&1009.61&0.00&1009.61\\
X_{1}&0.00&1009.61&0.00&0.02&0.00&-0.00\\
X_{2}&-0.00&0.00&1009.61&0.00&-0.00&-0.04\\
X_{1}^{2}&1009.61&0.02&0.00&2038565.01&-0.00&0.00\\
X_{1}X_{2}&0.00&0.00&-0.00&-0.00&0.00&-0.00\\
X_{2}^{2}&1009.61&-0.00&-0.04&0.00&-0.00&2038564.29\\
X_{1}^{3}&0.02&2038565.01&-0.00&-0.03&-0.00&0.00\\
X_{1}^{2}X_{2}&0.00&-0.00&0.00&-0.00&0.00&0.00\\
X_{1}X_{2}^{2}&-0.00&0.00&-0.00&0.00&0.00&-0.00\\
X_{2}^{3}&-0.04&-0.00&2038564.29&0.00&-0.00&-0.30\\
\end{block}
   \end{blockarray}\\ \nonumber
	\begin{blockarray}{ccccccc}
X_{1}^{3} & X_{1}^{2}X_{2} & X_{1}X_{2}^{2} & X_{2}^{3} & \text{ }\\
\begin{block}{cccc)ccc}
0.02&0.00&-0.00&-0.04&1\\
2038565.01&-0.00&0.00&-0.00&X_{1}\\
-0.00&0.00&-0.00&2038564.29&X_{2}\\
-0.03&-0.00&0.00&0.00&X_{1}^{2}\\
-0.00&0.00&0.00&-0.00&X_{1}X_{2}\\
0.00&0.00&-0.00&-0.30&X_{2}^{2}\\
4116200366.80&0.00&0.00&-0.00&X_{1}^{3}\\
0.00&0.00&-0.00&0.00&X_{1}^{2}X_{2}\\
0.00&-0.00&0.00&0.00&X_{1}X_{2}^{2}\\
-0.00&0.00&0.00&4116200384.54&X_{2}^{3}\\
\end{block}
   \end{blockarray}\\ \nonumber
\end{align}}

Since $A_{M}$ is flat, by Theorem \ref{above2} (2),  we can find the decomposition:
\begin{equation}
A_{M}=0 \cdot V_{3}V_{3}^{T}(0,-44.9351)+\frac12\cdot V_{3}V_{3}^{T}(-44.9351,0)+0 \cdot V_{3}V_{3}^{T}(44.9351,0)+\frac12\cdot V_{3}V_{3}^{T}(0,44.9351)\nonumber
\end{equation}
Moreover since $(-44.9351,0),(0,44.9351)\in \mathcal{V}_{f}$ we can conclude by Theorem \ref{calor} $(2)$, that:
\begin{equation}
P^{*}\in [-422.13, 1 ]\nonumber
\end{equation}

\end{ej}

\section{Conclusions}
For some polynomials, for example for nonnegative polynomials that are not sums of squares, sometimes one needs to solve a a moment relaxation of a very "big" degree to find the infimum. In that cases could be useful to use Algorithm $1$ \ref{upper} to get an idea of an "small" interval where the infimum belong to. In practice one can verify that the more iterations we apply in Algorithm $1$ \ref{upper} the better are the possibilities to get a flat solution and in consequence the better are the possibilities to find an interval where the infimum belong to.

The practical examples show that Algorithm 1 \ref{upper} works better as Algorithm \ref{upper2}. This is not surprising since the latest has the bigger challenge of minimize a polynomial over a semialgebraic set, in this case, over its real gradient variety. The additional request to the semidefinite program to fulfill the linear conditions associated to the real gradient variety reduces the set of flat feasible solutions. However, one can also try to solve the problem of minimize a polynomial over an arbitrary semialgebraic set using the same ideas as in Algorithm 2 \ref{upper2}.

\section{Software}
To find an optimal solution of the Moment relaxation and for the algorithm of extracting minimizers we have used the following softwares:
\begin{itemize}\label{software}
\item YALMIP: developed by J. Löfberg. It is a toolbox for Modeling and Optimization in MATLAB. Published in  the Journal Proceedings of the CACSD Conference in 2004. For more information see: http:\slash\slash yalmip.github.io\slash. 
\item SEDUMI: developed by J. F. Sturm. It is a toolbox for optimization over symmetric cones. Published in the Journal Optimization Methods and Software in 1999. For more information see: http:\slash\slash sedumi.ie.lehigh.edu\slash.
\item MATLAB and Statistics Toolbox Release 2016a, The MathWorks, Inc., Natick, Massachusetts, United States.
\end{itemize}

\end{document}